\documentclass[12pt,english]{amsart}

\usepackage{ae,aecompl}
\usepackage[T1]{fontenc}
\usepackage[latin9]{inputenc}
\usepackage[a4paper]{geometry}
\usepackage[active]{srcltx}
\usepackage{xcolor}
\usepackage{amstext}
\usepackage{amsthm}
\usepackage{amssymb}
\usepackage[all]{xy}
\PassOptionsToPackage{normalem}{ulem}
\usepackage{ulem}
\usepackage{hyperref}
\usepackage{babel}
\makeatletter

 \hypersetup{
           breaklinks=true,   % splits links across lines
           colorlinks=true,   % displays links as colored text instead of blocks
        }

\numberwithin{equation}{section}
\numberwithin{figure}{section}
\theoremstyle{plain}
\newtheorem{thm}{\protect\theoremname}[section]
  \theoremstyle{plain}
  \newtheorem{lem}[thm]{\protect\lemmaname}
  \theoremstyle{definition}
  \newtheorem{example}[thm]{\protect\examplename}
  \theoremstyle{remark}
  \newtheorem{rem}[thm]{\protect\remarkname}
  \theoremstyle{plain}
  \newtheorem{cor}[thm]{\protect\corollaryname}
  \theoremstyle{plain}
  \newtheorem{prop}[thm]{\protect\propositionname}
  
\providecommand{\corollaryname}{Corollary}
  \providecommand{\examplename}{Example}
  \providecommand{\lemmaname}{Lemma}
  \providecommand{\propositionname}{Proposition}
  \providecommand{\remarkname}{Remark}
\providecommand{\theoremname}{Theorem}

\@namedef{subjclassname@2020}{\textup{2020} Mathematics Subject Classification}
\makeatother

\def\quot{/\!\!/}
\def\C{\mathbb{C}}
\def\Z{\mathbb{Z}}
\def\R{\mathbb{R}}

\def\M{\mathcal{M}}
\def\Rep{\mathcal{R}}

\def\hom{\mathsf{Hom}}

\def\SL{\mathrm{SL}}
\def\GL{\mathrm{GL}}
\def\PGL{\mathrm{PGL}}

\def\SU{\mathrm{SU}}
\def\SO{\mathrm{SO}}
\def\Sp{\mathrm{Sp}}

\renewcommand{\textendash}{--}

\title[Mixed Hodge structures on character varieties]{Mixed Hodge structures on character varieties of nilpotent groups}

\author[C. Florentino]{Carlos Florentino}

\address{Departamento de Matem\'{a}tica, Faculdade de Ci\^{e}ncias, Univ. 
de Lisboa,  Edf. C6, Campo Grande 1749-016 Lisboa, Portugal}

\email{caflorentino@fc.ul.pt}

\author[S. Lawton]{Sean Lawton}
\address{Department of Mathematical Sciences, George Mason University, 4400 University Drive, Fairfax, Virginia 22030, USA}
\email{slawton3@gmu.edu}

\author[J. Silva]{Jaime Silva}

\address{Departamento Matem\'{a}tica,  ISEL - Instituto Superior de Engenharia de Lisboa, Rua Conselheiro Em\'{i}dio Navarro, 1, 1959-007 Lisboa, Portugal}

\email{jaime.a.m.silva@gmail.com}


\begin{document}

\begin{abstract}
Let $\hom^{0}(\Gamma,G)$ be the connected component of the identity
of the variety of representations of a finitely generated nilpotent
group $\Gamma$ into a connected reductive complex affine algebraic
group $G$. We determine the mixed Hodge structure on the representation variety $\hom^{0}(\Gamma,G)$
and on the character variety $\hom^{0}(\Gamma,G)\quot G$. We obtain
explicit formulae (both closed and recursive) for the mixed Hodge
polynomial of these representation and character varieties. 
\end{abstract}

\maketitle
\tableofcontents{}

\section{Introduction}

\thispagestyle{empty}

Let $K$ be a connected compact Lie group, and $\Gamma$ be a finitely
generated nilpotent group. The topology of the space of representations
$\hom(\Gamma,K)$ and of its conjugation quotient space $\hom(\Gamma,K)/K$
was considered by Ramras and Stafa in \cite{Sta,RS}, who obtained
explicit formulae for the Poincaré polynomials of their identity components
$\hom^{0}(\Gamma,K)$ and $\hom^{0}(\Gamma,K)/K$.

Let $G$ be the complexification of $K$, and consider now the affine
algebraic varieties $\mathcal{R}_{\Gamma}G:=\hom(\Gamma,G)$ and the
geometric invariant theoretic quotient by conjugation $\mathcal{M}_{\Gamma}G:=\mathcal{R}_{\Gamma}G\quot G$.
In this article we determine the mixed Hodge structures on the identity
components $\mathcal{R}_{\Gamma}^{0}G\subset\mathcal{R}_{\Gamma}G$
and $\mathcal{M}_{\Gamma}^{0}G\subset\mathcal{M}_{\Gamma}G$ and compute
their mixed Hodge polynomials, generalizing the formulas obtained
in \cite{RS} and in \cite{FS}.

We now describe more precisely our main results. A finitely generated
nilpotent group $\Gamma$ is said to have abelian rank $r$ if the
torsion free part of $\Gamma_{Ab}:=\Gamma/[\Gamma,\Gamma]$ has rank
$r$.  A connected reductive complex affine algebraic group $G$ will
be called a \textit{reductive $\mathbb{C}$-group}, and $T$, $W$
will stand, respectively, for a fixed maximal torus and the Weyl group
of $G$.

Recall that the mixed Hodge numbers $h^{k,p,q}(X)$ of a quasi-projective
variety $X$ are the dimensions of the $(p,q)$-summands of the natural
mixed Hodge structure (MHS) on $H^{k}(X,\mathbb{C})$. We say that
$X$ is of \emph{Hodge-Tate type} if $h^{k,p,q}(X)=0$ unless $p=q$. 

\begin{thm}
\label{thm:MHS}Let $\Gamma$ be a finitely generated nilpotent group
with abelian rank $r\geq1$, and $G$ a reductive $\mathbb{C}$-group.
Then, both $\mathcal{M}_{\Gamma}^{0}G$ and $\mathcal{R}_{\Gamma}^{0}G$
are of Hodge-Tate type. More concretely, the MHS on $\mathcal{M}_{\Gamma}^{0}G$
coincides with the one of $T^{r}/W$, where $W$ acts diagonally,
and the MHS of $\mathcal{R}_{\Gamma}^{0}G$ coincides with that of
$(G/T)\times_{W}T^{r}$. 
\end{thm}

\begin{rem}
In Theorem \ref{thm:MHS}, $W$ acts on $(G/T)\times T^{r}$ via the standard action on the homogeneous space $G/T$ and by simultaneous conjugation on $T^{r}$. The MHS on $G/T$ is the natural one coming from the full flag variety $G/B$, where $B\subset G$ is a Borel subgroup. We also note that the condition $r\geq 1$ in Theorem \ref{thm:MHS} is not vacuous since finite nilpotent groups have abelian rank $0$.  In fact, a nilpotent group $N$ has abelian rank $r\geq 1$ if and only if $|N|=\infty$.
\end{rem}

Now, consider the mixed Hodge polynomial of the algebraic variety
$X$, defined as:

\[
\mu_{X}(t,u,v)=\sum_{k,p,q\geq0}h^{k,p,q}(X)\,t^{k}u^{p}v^{q}\in\mathbb{Z}[t,u,v].
\]
Knowing the MHS of $\mathcal{M}_{\Gamma}^{0}G$ and $\mathcal{R}_{\Gamma}^{0}G$
allows for the explicit computation of their mixed Hodge polynomials,
as follows. Let $\mathfrak{t}$ denote the Lie algebra of the maximal
torus $T$, and recall that $W$ acts naturally on its dual $\mathfrak{t}^{*}$, as a reflection group. 

\begin{thm}
\label{thm:mu-formula}Let $\Gamma$ be a finitely generated nilpotent
group with abelian rank $r\geq1$, and $G$ a reductive $\mathbb{C}$-group
of rank $m$. Then, we have: 
\begin{equation}
\mu_{\mathcal{R}_{\Gamma}^{0}G}\left(t,u,v\right)=\frac{1}{|W|}\prod_{i=1}^{m}(1-(t^{2}uv)^{d_{i}})\sum_{g\in W}\,\frac{\det\left(I+tuv\,A_{g}\right)^{r}}{\det\left(I-t^{2}uv\,A_{g}\right)}\label{eq:mu-Rgamma}
\end{equation}
and 
\[
\mu_{\mathcal{M}_{\Gamma}^{0}G}\left(t,u,v\right)=\frac{1}{|W|}\sum_{g\in W}\,\det\left(I+tuv\,A_{g}\right)^{r},
\]
where $d_{1},\ldots,d_{m}$ are the characteristic degrees of $W$ $($see definition in Section \ref{sec:mhs-smooth}$),$ and $A_{g}$ is the action of $g\in W$ on $H^{1}(T,\mathbb{C})\cong\mathfrak{t}^{*}$. 
\end{thm}

We now outline the proofs of these theorems. Using the main results of \cite{BS}, we start by considering $\mathbb{Z}^{r}$, the free
part of $\Gamma_{Ab}$. Let $K$ be a maximal compact
subgroup of $G$. Considering the deformation retractions obtained
in \cite{FL2} for $\mathcal{M}_{\mathbb{Z}^{r}}G$, and in \cite{PeSo}
for $\mathcal{R}_{\mathbb{Z}^{r}}G$, we are then reduced to describing
the cohomology of the compact spaces $\mathcal{R}_{\mathbb{Z}^{r}}^{0}K:=\hom^{0}(\Gamma,K)$
and $\hom^{0}(\Gamma,K)/K$.

A priori, there is no reason for these compact spaces to have MHSs
on their cohomology groups. In \cite{Ba} the rational cohomology
ring of $\mathcal{R}_{\mathbb{Z}^{r}}^{0}K$ is shown to be the Weyl
group invariants of $(K/T_{K})\times T_{K}^{r}$ where $T_{K}=T\cap K\subset K$
is a maximal torus. $(K/T_{K})\times_{W}T_{K}^{r}$ is a desingularization
of $\mathcal{R}_{\mathbb{Z}^{r}}^{0}K$, and is homotopic to the space
$(G/T)\times_{W}T^{r}$. Given the natural MHSs on $G/T$, $T$, and
on the classifying space $BT$, in the context of equivariant cohomology,
we conclude that both $\mathcal{R}_{\mathbb{Z}^{r}}^{0}G$ and $\mathcal{M}_{\mathbb{Z}^{r}}^{0}G$
are of Hodge-Tate type.

The formula for $\mathcal{M}_{\Gamma}^{0}G$ then follows from the
one in \cite{FS}. To get the formula for $\mathcal{R}_{\Gamma}^{0}G$
we observe, as in \cite[Section 5]{RS}, that the graded cohomology
ring of $\mathcal{R}_{\mathbb{Z}^{r}}K$ is a regrading of the cohomology
ring of the torus $T^{r}$. Using representation theory, analogous
to what is done in \cite{KT}, we determine the regrading explicitly
to obtain Formula \eqref{eq:mu-Rgamma}.

The main results are proved in Sections \ref{sec:mhs-rep} and \ref{sec:mhs-char},
after a brief review of relevant facts about mixed Hodge structures
and character varieties in Section \ref{sec-charvar}. In Section
\ref{sec:irr-comp}, we show that although the path-component of the
identity is a union of algebraic components and the mixed Hodge structure
is determined by the torus component (irreducible by \cite{Sik}),
there is in fact only one irreducible component through the identity.
This follows by closely analyzing the main proof in \cite{FL2}. Moreover,
we give a description of the singular locus of these moduli spaces $\mathcal{M}_{\Z^r}^{0}G$ in the
cases of classical $G$ (expanding on work of \cite{Sik}). We also obtain interesting
number-theoretic results in Subsection \ref{subsec:poly-count}.  In particular, we show that 
$\mathcal{M}_{\Z^r}^{0}G$ are ``polynomial count" and compute the number of $\mathbb{F}_q$-points of these varieties where $\mathbb{F}_q$ is a field of order $q$. The last
section (Section \ref{sec:exotic}) applies our results to examples of character and representation varieties with ``exotic components'' considered in \cite{ACG};
here, the group $G$ is of the form $\SL(p,\mathbb{C})^{m}/\mathbb{Z}_{p}$
for a prime $p$.

\subsection*{Acknowledgments}

We thank Dan Ramras for pointing out a useful reference. Lawton is partially supported by a Collaboration Grant
from the Simons Foundation, and thanks IHES for hosting him in 2021 when this work was advanced. Florentino was supported by CMAFcIO (University of Lisbon), under project UIDB/04561/2020, FCT Portugal. Lastly, we thank the referee for helpful comments.

\section{Character Varieties and Mixed Hodge Structures}

\label{sec-charvar}

\subsection{Character Varieties}

Let $G$ be a connected reductive complex affine algebraic group.
As mentioned earlier, we will say $G$ is a \emph{reductive $\C$-group}
in abbreviation. Let $\Gamma$ be a finitely generated group. The
set of homomorphisms $\rho:\Gamma\to G$ has the structure of an affine
algebraic variety over $\mathbb{C}$ (not necessarily irreducible);
the generators of $\Gamma$ are translated into elements of $G$ satisfying
algebraic relations determined by the relations of $\Gamma$. Since
$G$ admits a faithful representation $G\hookrightarrow\mathrm{GL}(n,\C)$
for some $n$, we will sometimes refer to $\rho$ as a $G$-representation
of $\Gamma$, or simply a representation of $\Gamma$ when the context
is clear.

We have two naturally defined varieties: the $G$-representation variety
of $\Gamma$, 
\[
\mathcal{R}_{\Gamma}G:=\hom(\Gamma,G),
\]
and the $G$-character variety of $\Gamma$, 
\[
\mathcal{M}_{\Gamma}G:=\hom\left(\Gamma,G\right)\quot G,
\]
which is the affine geometric invariant theoretic (GIT) quotient under
the conjugation action of $G$ on $\mathcal{R}_{\Gamma}G$.

We endow $\mathcal{R}_{\Gamma}G$ with the Euclidean topology coming
from a choice of $r$ generators of $\Gamma$ and the natural embedding
$\hom(\Gamma,G)\hookrightarrow G^{r}\subset\mathbb{C}^{rn^{2}}$,
for appropriate $n$. Hence, $\mathcal{M}_{\Gamma}G$ is naturally
endowed with a Hausdorff topology, as the GIT quotient identifies
orbits whose closures intersect (see \cite[Theorem 2.1]{FL2} for
a precise statement). However, when speaking of irreducible components
we refer to the Zariski topology.

We note that $\mathcal{M}_{\Gamma}G$ is homotopic to the non-Hausdorff
(conjugation) quotient space $\mathcal{R}_{\Gamma}G/G$ by \cite[Proposition 3.4]{FLR},
and so any homotopy invariant property of either $\M_{\Gamma}(G)$
or $\Rep_{\Gamma}G/G$ applies to the other.

\subsection{Mixed Hodge Structures}

In this subsection we summarize facts about mixed Hodge structures;
details can be found in \cite{De1,De2,PS,Vo}. The singular cohomology
of a complex variety $X$ is endowed with a decreasing \textit{Hodge
filtration} $F_{\bullet}$: 
\[
H^{k}\left(X,\mathbb{C}\right)=F_{0}\supseteq\cdots\supseteq F_{k+1}=0
\]
that generalizes the same named filtration for smooth complex projective
varieties. In general, the graded pieces of this filtration do not
constitute a pure Hodge structure. However, the rational cohomology
of these varieties admits an increasing \textit{Weight filtration}:
\[
0=W^{-1}\subseteq\cdots\subseteq W^{2k}=H^{k}\left(X,\mathbb{Q}\right),
\]
and the Hodge filtration induces a pure Hodge structure on the graded
pieces of its complexification, denoted $W_{\mathbb{C}}^{\bullet}$.
The triple $\left(H^{k}\left(X,\mathbb{C}\right),F_{\bullet},W_{\mathbb{C}}^{\bullet}\right)$
constitutes a \textit{mixed Hodge structure }(MHS) over $\C$, and
we denote the graded pieces of the associated decomposition by: 
\begin{eqnarray*}
H^{k,p,q}\left(X,\mathbb{C}\right) & = & Gr_{F}^{p}Gr_{p+q}^{W_{\mathbb{C}}}H^{k}\left(X,\mathbb{C}\right).
\end{eqnarray*}
Their dimensions, called \emph{mixed}\textit{ Hodge numbers} $h^{k,p,q}(X):=\dim_{\mathbb{C}}H^{k,p,q}(X,\mathbb{C})$,
are encoded in the polynomial: 
\[
\mu_{X}(t,u,v)=\sum_{k,p,q\geq0}h^{k,p,q}(X)\,t^{k}u^{p}v^{q}\in\mathbb{Z}[t,u,v],
\]
called the \textit{mixed Hodge polynomial} of $X$. This polynomial
reduces to the Poincaré polynomial of $X$, by setting $u=v=1$. These
constructions can also be reproduced for compactly supported cohomology,
yielding a similar decomposition into pieces denoted $H_{c}^{k,p,q}(X,\C)$.

When the variety $X$ is smooth and projective the Hodge structure on $H^{*}(X,\C)$
is pure, that is: $h^{k,p,q}\neq0\implies k=p+q$. We are also interested
in two other types of MHS that can be read from its Hodge numbers.
We say that $X$ is \textit{balanced } (see \cite{LMN}) or of\textit{ Hodge-Tate type}
if $h^{k,p,q}\neq0\implies p=q$. For those varieties that further
satisfy $h^{k,p,q}\neq0\implies k=p=q$ we call them \textit{round} (see \cite{FS}).

Recall that MHSs satisfy the Künneth theorem, so that, for the cartesian
product $X\times Y$ of varieties, we have: 
\begin{equation}
\mu_{X\times Y}=\mu_{X}\,\mu_{Y}.\label{eq:product}
\end{equation}

Also important for this paper is the behavior of these structures
under an algebraic action of a finite group. If $F$ is a finite group
acting algebraically on a complex variety $X$, the induced action
on the cohomology respects the mixed Hodge decomposition. Moreover,
one can recover the mixed Hodge structure on the quotient by: 
\begin{eqnarray}
H^{k,p,q}(X/F,\mathbb{C}) & \cong & H^{k,p,q}(X,\mathbb{C})^{F}.\label{eq:Groiso}
\end{eqnarray}
Then, the types of mixed Hodge structures on the quotient $X/F$ have
similar properties to that of $X$. In particular, if $X$ is pure,
balanced or round, respectively, so is $X/F$. The situation is even
easier when $G$ is an algebraic group and $F$ is a finite subgroup
acting by left translation. 
\begin{lem}
\label{lem:MHS}Let $G$ be an algebraic group and $F$ a finite subgroup.
Then the MHS on $G$ and on $G/F$ coincide. 
\end{lem}

\begin{proof}
This follows from the fact that the $F$-action on the MHS of $G$
is trivial, as shown in \cite[Section 6]{DiLe} (see also \cite{LMN}). Intuitively, the idea is
that the action of $F$ extends to the action of a connected group. 
\end{proof}
Another important invariant related to the MHS of $X$ is the $E$-polynomial,
obtained by specializing $\mu_{X}$ to $t=-1$: $E_{X}(u,v):=\mu_{X}(-1,u,v)$.
Then the Euler characteristic of $X$ is obtained as $\chi(X)=\mu_{X}(-1,1,1)$.
We will also consider the compactly supported version of $E_{X}$, also
called the \emph{Serre polynomial}: 
\[
E_{X}^{c}(u,v):=\sum_{k,p,q\geq0}\left(-1\right)^{k}h_{c}^{k,p,q}(X)\,u^{p}v^{q}\in\mathbb{Z}[t,u,v],
\]
where $h_{c}^{k,p,q}(X)=\dim H_{c}^{k,p,q}(X,\C)$.

Let $K\left(\mathcal{V}ar_{\mathbb{C}}\right)$ be the \textit{Grothendieck
ring of varieties} over $\mathbb{C}$. Additively, this is a ring
generated by isomorphism classes of algebraic varieties modulo the
excision relation: if $Y\hookrightarrow X$ is a closed subvariety,
then in $K\left(\mathcal{V}ar_{\mathbb{C}}\right)$ we identify: 
\begin{eqnarray*}
\left[X\right] & = & \left[Y\right]+\left[X\backslash Y\right].
\end{eqnarray*}
The product in $K\left(\mathcal{V}ar_{\mathbb{C}}\right)$ is given
by cartesian product: $\left[X\right]\cdot\left[Y\right]:=\left[X\times Y\right]$.
The Serre polynomial
and the Euler characteristic are examples of \textit{motivic measures}, that is, maps from the objects of $\mathcal{V}ar_{\mathbb{C}}$ to a ring
that factors through the Grothendieck ring of varieties.

\section{Irreducible Components}

\label{sec:irr-comp}

For many groups $\Gamma$, $\mathcal{R}_{\Gamma}G$ is \emph{not irreducible}
and/or \emph{not path-connected}, and so the same happens with $\mathcal{M}_{\Gamma}G$.
Recall that path-connected algebraic varieties need not be irreducible,
and that irreducible algebraic varieties (over $\mathbb{C}$) are
necessarily path-connected.

Path-components of $\mathcal{R}_{\Gamma}G$ are sometimes related
to path-components of $\mathcal{R}_{\Gamma}K$ for a maximal compact
subgroup $K\subset G$. For example, for a finitely generated free
group $F_{r}$, $\Rep_{F_{r}}G\cong G^{r}$ and $\Rep_{F_{r}}K\cong K^{r}$
and so there is a $\pi_{0}$-bijection by the (topological) polar
decomposition: $G\cong K\times\R^{n}$, for $n=\dim_{\R}K$. Much
more non-trivially, there is a strong deformation retraction from
$\Rep_{\Gamma}G$ to $\Rep_{\Gamma}K$ for $\Gamma$ finitely generated
and nilpotent by \cite{Be}; see \cite{PeSo} for the abelian case.
And thus, there is a bijection between path-components in these cases
as well. 
\begin{example}
\label{exa:disconnected}Let $\Gamma=\mathbb{Z}^{2}$ and $K=\mathrm{SO}(3)$.
Suppose $\rho\in\hom(\Gamma,K)$ is given by the pair of commuting
matrices $\mbox{diag}(1,-1,-1)$ and $\mbox{diag}(-1,-1,1)$. Then,
these matrices cannot be simultaneously conjugated, within $K$, to
the same maximal torus of $\SO(3)$. This implies that $\hom(\Gamma,K)$
is not path-connected, since the collection of pairs that can be simultaneously
conjugated into a given maximal torus forms a disjoint path-component.
Thus, by the discussion above, $\hom(\mathbb{Z}^{2},\PGL(2,\mathbb{C}))$
is also not connected, as $\PGL(2,\mathbb{C})\cong\mathrm{SO}(3,\mathbb{C})$
is the complexification of $\mathrm{SO}(3)$. 
\end{example}

Let us denote by 
\[
\mathcal{R}_{\Gamma}^{0}G:=\hom^{0}(\Gamma,G),
\]
and by 
\[
\mathcal{M}_{\Gamma}^{0}G:=\hom^{0}\left(\Gamma,G\right)\quot G,
\]
the path-connected components of the identity representation in $\mathcal{R}_{\Gamma}G$
and in $\mathcal{M}_{\Gamma}G$, respectively. In some cases, $\mathcal{R}_{\Gamma}^{0}G$
and $\mathcal{M}_{\Gamma}^{0}G$ are irreducible varieties; but they
are always a finite union of irreducible varieties.

\subsection{The Torus Component}

An interesting case is that of a finitely presentable group $\Gamma$
whose abelianization is free, that is 
\[
\Gamma_{Ab}:=\Gamma/[\Gamma,\Gamma]\cong\Z^{r},
\]
for some $r\in\mathbb{N}$. Examples in this class of groups include
``exponent canceling groups'' (see \cite{LaRa}) which are those
that admit presentations such that in all relations the exponents
of each generator add up to zero; such as right angled Artin groups
(abbreviated RAAGs), and fundamental groups of closed orientable surfaces.

For these groups, since $\Gamma\to\Gamma_{Ab}\cong\Z^{r}$ is surjective,
we can consider the following sequence: 
\[
T^{r}\cong\hom(\mathbb{Z}^{r},T)\hookrightarrow\hom(\Gamma_{Ab},G)\hookrightarrow\hom(\Gamma,G)\twoheadrightarrow\mathcal{M}_{\Gamma}G.
\]
Let us denote by $\mathcal{M}_{\Gamma}^{T}G\subset\mathcal{M}_{\Gamma}G$
the image of the composition above and call it the \emph{torus component}. It follows that $\mathcal{M}_{\Gamma}^{T}G$ is an irreducible subvariety of
$\mathcal{M}_{\Gamma}G$, being the image of the irreducible variety $T^r$ under a morphism.  In the case when $\Gamma=\Z^r$, $\mathcal{M}_{\Gamma}^{T}G$ is an irreducible component of 
$\mathcal{M}_{\Gamma}G$ by \cite[Theorem 2.1]{Sik}.

Obviously, the identity representation ($\rho(\gamma)=e$ for all
$\gamma\in\Gamma$, $e\in G$ being the identity) belongs to $\mathcal{M}_{\Gamma}^{T}G$
since it comes from the identity representation in $\hom(\Gamma,T)$.
Since $\mathcal{M}_{\Gamma}^{T}G$ is path-connected (being irreducible
over $\mathbb{C}$), we conclude that $\mathcal{M}_{\Gamma}^{T}G\subset\mathcal{M}_{\Gamma}^{0}G$.
We observe that there are pairs $(\Gamma,G)$ where the varieties
$\mathcal{M}_{\Gamma}^{T}G$ and $\mathcal{M}_{\Gamma}^{0}G$ agree,
and others where they do not. 
\begin{example}
\label{exa:zero-and-T}When $G$ is abelian (we always assume connected),
it is clear that $\mathcal{M}_{\Gamma}^{T}G=\mathcal{M}_{\Gamma}^{0}G$.
For an example where they disagree, let $\Gamma=\Gamma_{\angle}$
be the ``angle RAAG'' associated with a path graph with 3 vertices,
considered in \cite{FL4}. Then, even for a low dimensional group
such as $G=\SL(2,\mathbb{C})$ we have that $\mathcal{M}_{\Gamma}^{0}G$
has 3 irreducible components, one being $\mathcal{M}_{\Gamma}^{T}G$
and 2 extra ones. Moreover, for the case $G=\SL(3,\mathbb{C})$ there
are components in $\mathcal{M}_{\Gamma}^{0}G$ which have higher dimension
than the dimension of $\mathcal{M}_{\Gamma}^{T}G$. 
\end{example}

\begin{rem}
\label{rem:components}One can also ask if the identity representation
is contained in a single irreducible component of $\mathcal{M}_{\Gamma}^{0}G$.
This also fails for $\mathcal{M}_{\Gamma_{\angle}}^{0}(\SL(2,\mathbb{C}))$,
as shown in \cite{FL4}. 
\end{rem}

\subsection{The Free Abelian Case}

As seen in Examples \ref{exa:disconnected}, \ref{exa:zero-and-T}
and Remark \ref{rem:components}, the comparison between the varieties
$\mathcal{M}_{\Gamma}G$, $\mathcal{M}_{\Gamma}^{0}G$ and $\mathcal{M}_{\Gamma}^{T}G$
for general $\Gamma$ and $G$ is far from being trivial.

We now show that $\mathcal{M}_{\Gamma}^{0}G=\mathcal{M}_{\Gamma}^{T}G$
when $\Gamma$ is free abelian, for all $G$. In this situation, we
are dealing with representations defined by elements of $G$ that
pairwise commute. The following theorem generalizes Remark 2.4 in
\cite{Sik}, and completely answers a question raised in \cite[Problem 5.7]{FL2}. 
\begin{thm}
\label{thm:Tequals0}For every $r\in\mathbb{N}$ and reductive $\mathbb{C}$-group
$G$, $\mathcal{M}_{\mathbb{Z}^{r}}^{T}G=\mathcal{M}_{\mathbb{Z}^{r}}^{0}G$. 
\end{thm}

\begin{proof}
Let $K$ be a fixed maximal compact subgroup of $G$, and let $T_{K}=T\cap K$.
Then $T_{K}\subset K$ is a maximal torus in $K$.
Tom Baird \cite{Ba} considered the \emph{compact character variety}
\[
\mathcal{N}_{\mathbb{Z}^{r}}K:=\hom(\mathbb{Z}^{r},K)/K,
\]
and showed that: 
\begin{enumerate}
\item The (path) connected component of the identity $\mathcal{N}_{\mathbb{Z}^{r}}^{0}K\subset\mathcal{N}_{\mathbb{Z}^{r}}K$
coincides with the space of conjugation classes of representations
\[
\hom^{T_{K}}(\mathbb{Z}^{r},K)/K,
\]
where $\hom^{T_{K}}(\mathbb{Z}^{r},K)$ denotes the representations
$\rho$ whose $r$ evaluations $\rho(\gamma_{i})$ can be simultaneously
conjugated into the maximal torus $T_{K}$. 
\item $\mathcal{N}_{\mathbb{Z}^{r}}^{0}K$ is homeomorphic to the quotient
$T_{K}^{r}/W$, where $W=N_{K}(T_{K})/T_{K}$ is the Weyl group associated
to $T_{K}$. 
\end{enumerate}
By \cite[Theorem 1.1]{FL2}, there is a strong deformation retraction
from $\mathcal{M}_{\mathbb{Z}^{r}}G$ to $\mathcal{N}_{\mathbb{Z}^{r}}K$
which (by continuity) restricts to one from $\mathcal{M}_{\mathbb{Z}^{r}}^{0}G$
to $\mathcal{N}_{\mathbb{Z}^{r}}^{0}K$.

Let $[\rho]\in\mathcal{M}_{\mathbb{Z}^{r}}^{0}G$. Then there exists
a commuting tuple $(g_{1},\ldots,g_{r})$ in $G^{r}$ such that $[\rho]=[(g_{1},...,g_{r})]$.
By \cite[Proposition 3.1]{FL2} we can assume that each element $g_{i}\in G$
is semisimple. The strong deformation retraction (SDR), which is $G$-conjugate
equivariant, provides a path $\rho_{t}$ from this tuple to a commuting
tuple in $G_{K}^{r}$ where $G_{K}=\{gkg^{-1}\ |\ g\in G,\ k\in K\}$.
With respect to an embedding $G\hookrightarrow\SL(n,\C)$, which preserves
semisimplicity, this SDR is given by the eigenvalue retraction defined
by deforming $\ell e^{i\theta}$ to $e^{i\theta}$ by sending $\ell$
to 1. By \cite[Lemma 3.4]{FL2}, there exists a single element $g_{0}$
in $G$ that will conjugate the resulting $r$-tuple in $G_{K}^{r}$
to be in $K^{r}$. And by Baird's result, and the fact that we remain
in the identity component by continuity, we know there is a single
element $k_{0}$ in $K$ that we can conjugate the resulting tuple
in $K^{r}$ so it is in $T_{K}^{r}$. Let $h_{0}:=k_{0}g_{0}$, and
consider the conjugated reverse path $\psi_{t}:=h_{0}\rho_{1-t}h_{0}^{-1}$.
This path begins in $T_{K}^{r}$ (by definition) and ends in $T^{r}$ since eigenvalue retraction deforms $T$ to $T_K$ and hence the reverse path takes elements in $T_K$ and maps them to $T$.
Since $\psi_{1}=h_{0}\rho_{0}h_{0}^{-1}=(h_{0}g_{1}h_{0}^{-1},\ldots, h_{0}g_{r}h_{0}^{-1})$,
we conclude that 
\[
[\rho]=[(g_{1},...,g_{r})]=[(h_{0}g_{1}h_{0}^{-1},\ldots, h_{0}g_{r}h_{0}^{-1})]=[\psi_{1}]
\]
is in $T^{r}/W$, as required. 
\end{proof}

\section{Mixed Hodge Structure on \texorpdfstring{$\hom^{0}(\Gamma,G)$}{Hom(F,G)} }

\label{sec:mhs-rep}

Here we prove the statements in Theorems \ref{thm:MHS} and \ref{thm:mu-formula}
that concern the connected component $\mathcal{R}_{\Gamma}^{0}G$
of the trivial representation in the representation variety $\mathcal{R}_{\Gamma}G=\hom(\Gamma,G)$.

Let us first describe the situation for $\Gamma\cong\mathbb{Z}^{r}$.
Consider, as in the proof of Theorem \ref{thm:Tequals0}, the compact
character variety 
\[
\mathcal{N}_{\mathbb{Z}^{r}}K=\hom(\mathbb{Z}^{r},K)/K,
\]
where $K$ is a fixed maximal compact subgroup of $G$. Recall our
convention that $T_{K}=T\cap K$ where $T$ is a maximal torus in
$G$. Baird \cite{Ba} showed that the isomorphism $\mathcal{N}_{\mathbb{Z}^{r}}^{0}K\cong T_{K}^{r}/W$,
is part of a natural $K$-equivariant commutative diagram:

\[
\xymatrix{(K/T_{K})\times_{W}T_{K}^{r}\ar[d]_{\pi_{K}}\ar[r]^{\varphi_{K}} & \hom^{0}(\mathbb{Z}^{r},K)\ar[d]^{\pi_{K}}\\
T_{K}^{r}/W\ar[r]^{\cong} & \mathcal{N}_{\mathbb{Z}^{r}}^{0}K,
}
\]
where $\varphi_{K}$ is a desingularization of $\hom^{0}(\mathbb{Z}^{r},K)$
which induces an isomorphism in cohomology, and the vertical maps
are the quotient maps by $K$-conjugation.

Passing to the complexification, there is an analogous $G$-equivariant
commutative diagram: 
\[
\xymatrix{(G/T)\times_{W}T^{r}\ar[d]_{\pi_{G}}\ar[r]^{\varphi_{G}} & \hom^{0}(\mathbb{Z}^{r},G)\ar[d]^{\pi_{G}}\\
T^{r}/W\ar[r]^{\chi} & \mathcal{M}_{\mathbb{Z}^{r}}^{0}G.
}
\]
There are some notable differences from the compact case: 
\begin{enumerate}
\item $\chi$ is bijective, birational, and a normalization map (Corollary
\ref{cor-normalization}), 
\item $\chi$ is not generally known to be an isomorphism, although it is
when $G$ is of classical type (Corollary \ref{cor-classical}), 
\item $\varphi_{G}$ is not even surjective, hence not a desingularization
morphism (we will say more about this below), 
\item $\pi_{G}$ is the GIT quotient map (with respect to the $G$-conjugation
action). 
\end{enumerate}
Despite these differences, we will show that the mixed Hodge structures
of $T^{r}/W$ and $\M_{\Z^{r}}^{0}G$ coincide, as do those of $(G/T)\times_{W}T^{r}$
and $\hom^{0}(\Z^{r},G)$.

\subsection{Mixed Hodge Structures on a Smooth Model of $\mathcal{R}_{\mathbb{Z}^{r}}^{0}G$}\label{sec:mhs-smooth}

The above discussion suggests to consider the smooth irreducible algebraic
variety: 
\[
S_{r}G:=(G/T)\times_{W}T^{r},
\]
whose MHS we now determine. The natural MHS on $G/T$ is the one of
the full flag variety $G/B$, where $B$ is a Borel subgroup, which
has well-known cohomology. Indeed, it is a classical fact that there
is an identification $K/T_{K}\cong G/B$. On the other hand, $K/T_{K}\hookrightarrow G/T$
is a strong deformation retraction (see for example \cite[Theorem 10]{BFLL}),
which provides isomorphisms of cohomology spaces: 
\[
H^{*}(G/B)\cong H^{*}(K/T_{K})\cong H^{*}(G/T).
\]
Since $T$ is contained in a certain Borel subgroup, there is a surjective
algebraic map $\varphi:G/T\to G/B$ which upgrades the above isomorphism
to an isomorphism of MHSs, $H^{*}(G/B)\cong H^{*}(G/T)$, since the restriction of $\varphi$ to $K/T_{K}$ is the map that induces the isomorphism $K/T_{K}\cong G/B$. 

Recall that the Weyl group $W$ acts on $\mathfrak{t}^*$, the dual of the Lie algebra of $T$. By the Shephard-Todd theorem \cite{ShTo}, the ring of $W$-invariants $\C[\mathfrak{t}^*]^W$ is a polynomial ring generated by homogeneous generators of degrees $d_{1},\ldots,d_{m}$ called the {\it characteristic degrees} of $W$ ($m=\dim \mathfrak{t}$). These
are well-known for all Weyl groups of simple $G$ (isomorphic to the Weyl groups of simple $K$), and can be consulted in \cite[Table 1]{RS} or in \cite[Page 7]{KT}.

\begin{thm}
\label{thm:SrG} Let $m$ be the rank of $G$ and $d_{1},\ldots,d_{m}$
be the characteristic degrees of $W$. The variety $S_{r}G$ is of Hodge-Tate
type and its mixed Hodge polynomial is given by: 
\[
\mu_{S_{r}G}\left(t,u,v\right)=\frac{1}{|W|}\prod_{i=1}^{m}(1-(t^{2}uv)^{d_{i}})\sum_{g\in W}\,\frac{\det\left(I+tuv\,A_{g}\right)^{r}}{\det\left(I-t^{2}uv\,A_{g}\right)}.
\]
\end{thm}

\begin{proof}
Since mixed Hodge structures respect the Künneth formula, from $H^{*}(G/B)\cong H^{*}(G/T)$,
we get an isomorphism: 
\[
H^{*}(S_{r}G)\cong[H^{*}(G/T)\otimes H^{*}(T^{r})]^{W}\cong[H^{*}(G/B)\otimes H^{*}(T^{r})]^{W},
\]
of mixed Hodge structures, where the superscript means that we are
considering the $W$-invariant subspace. Since the full flag variety
$G/B$ is smooth and projective, its cohomology has a pure Hodge structure.
Moreover, there is an isomorphism 
\[
H^{*}(G/B)\cong H^{*}(BT)_{W}
\]
where $BT\cong(BS^{1})^{m}$ is the classifying space of $T$, and
$H^{*}(BT)_{W}$ is the algebra of co-invariants under the $W$-action
on $H^{*}(BT)$. Also, $H^{*}(BT)$ is a polynomial ring $\mathbb{C}[x_{1},\ldots,x_{m}]$
where each $x_{i}$ has triple grading $(2,1,1)$, since $BT$ can
be identified with $(\mathbb{C}P^{\infty})^{m}$ (in particular, it
has pure cohomology). By a classical theorem of Borel (see \cite{Re}
for a modern treatment), there is an isomorphism: 
\[
H^{*}(BT)_{W}\cong\mathbb{C}[x_{1},\ldots,x_{m}]/(\sigma_{1},\ldots,\sigma_{m}),
\]
where the $\sigma_{i}$ are the homogeneous generators of the ring
of $W$-invariants $H^{*}(BT)^{W}$, with degrees $(2d_{i},d_{i},d_{i})$.

From the above, and the fact that $\sigma_{1},\ldots,\sigma_{m}$
are $W$-invariants, we obtain: 
\[
H^{*}(S_{r}G)\cong[H^{*}(G/B)\otimes H^{*}(T^{r})]^{W}\cong[H^{*}(BT)\otimes H^{*}(T^{r})]^{W}/(\sigma_{1},\ldots,\sigma_{m}).
\]
Now, the mixed Hodge polynomial $\mu_{X}(t,u,v)$ of a variety $X$
is the Hilbert series of its cohomology $H^{*}(X)$ with the triple
grading given by its mixed Hodge structure. Denote by $\mathfrak{H}(A)$
the Hilbert series of a graded algebra $A$, in the variable $x$.
It is a standard result that, if $a\in A$ is not a zero divisor,
then 
\[
\mathfrak{H}(A/(a))=\mathfrak{H}(A)\,(1-x^{d}),
\]
where $d$ is the degree of $a$. Applied to our case, and since $\sigma_{1},\ldots,\sigma_{m}$
form a regular sequence (see \cite{KT}), we get the equality of Hilbert series in the
three variables $t,u,v$:

\[
\mathfrak{H}(H^{*}(S_{r}G))=\mathfrak{H}([H^{*}(BT)\otimes H^{*}(T^{r})]^{W})\,\prod_{i=1}^{m}(1-(t^{2}uv)^{d_{i}}).
\]
The result thus follows from: 
\begin{equation}\label{eq:hilbert}
\mathfrak{H}([H^{*}(BT)\otimes H^{*}(T^{r})]^{W})=\frac{1}{|W|}\sum_{g\in W}\,\frac{\det\left(I+tuv\,A_{g}\right)^{r}}{\det\left(I-t^{2}uv\,A_{g}\right)}.
\end{equation}
Formula \eqref{eq:hilbert} is obtained by applying Corollary \ref{cor:Hilbert}
below with $V_{0}=H^{2,1,1}(BT)$ and $V_{1}=\cdots=V_{r}=H^{1,1,1}(T)$,
since $H^{*}(BT)=S^{\bullet}V_{0}$ and $T^{r}$ has round cohomology
generated in degrees $(1,1,1)$: $H^{*}(T^{r})=\wedge^{\bullet}H^{1,1,1}(T^{r})\cong(\wedge^{\bullet}H^{1,1,1}(T))^{\otimes r}$. 
\end{proof}

Recall some definitions and facts from the theory of representations
of finite groups. If $V=\oplus_{k\geq0}V^{k}$ is a graded $\mathbb{C}$-vector
space (possibly infinite dimensional, but with finite dimensional
summands), and $g:V\to V$ is a linear map that preserves the grading,
define the \emph{graded-character} of $g$ by: 
\[
\chi_{g}(V):=\sum_{k\geq0}\text{tr}(g|_{V^{k}})\,x^{k}\in\mathbb{C}[[x]].
\]
It is additive and multiplicative, under direct sums and tensor products,
respectively: 
\[
\chi_{g}(V_{1}\oplus V_{2})=\chi_{g}(V_{1})+\chi_{g}(V_{2}),\quad\quad\chi_{g}(V_{1}\otimes V_{2})=\chi_{g}(V_{1})\chi_{g}(V_{2}).
\]

A linear map $g:V\to V$ induces linear maps on the direct sums of
all symmetric powers $S^{\bullet}V:=\oplus_{j\geq0}S^{j}V$, and of
all exterior powers $\wedge^{\bullet}V:=\oplus_{j\geq0}^{\dim V}\wedge^{j}V$.
Note that $S^{\bullet}V$ is graded, with elements of $S^{j}V$ and
$\wedge^{j}V$ having degree $j\delta$, when $V$ is pure of degree
$\delta$. We need to consider two important cases, whose proofs are
standard (see, for instance, \cite[page 69]{JPSer}). 
\begin{lem}
\label{lem:sym-and-wedge} Let $V$ be a vector space, whose elements
are all of degree $\delta$, and $g:V\to V$ a linear map. Then, we
have: 
\[
\chi_{g}(S^{\bullet}V)=\frac{1}{\det(I-x^{\delta}g)},\quad\quad\chi_{g}(\wedge^{\bullet}V)=\det(I+x^{\delta}g).
\]
\end{lem}

Now, suppose that a finite group $F$ acts on $V$ preserving the
grading. Recall that the Hilbert-Poincaré series of the graded vector
space $V^{F}$, of $F$-invariants in $V$, can be computed as: 
\begin{equation}
\mathfrak{H}(V^{F})=\frac{1}{|F|}\sum_{g\in F}\chi_{g}(V).\label{eq:PH-series}
\end{equation}
Since all the above constructions are valid for triply graded vector
spaces, and characters are multiplicative under tensor products, the
following is an immediate consequence of Lemma \ref{lem:sym-and-wedge}. 
\begin{cor}
\label{cor:Hilbert}If $V_{i}$, $i=0,\ldots,r$ are finite dimensional
representations of a finite group $F$, then the triply graded Hilbert-Poincaré
series in $t,u,v$ is: 
\[
\mathfrak{H}([S^{\bullet}V_{0}\otimes\wedge^{\bullet}V_{1}\otimes\cdots\otimes\wedge^{\bullet}V_{r}]^{F})=\frac{1}{|F|}\sum_{g\in F}\frac{\prod_{i=1}^{r}\det(I+t^{a_{i}}u^{b_{i}}v^{c_{i}}g)}{\det(I-t^{a_{0}}u^{b_{0}}v^{c_{0}}g)},
\]
where each $V_{i}$ has pure triple degree $(a_{i},b_{i},c_{i})$,
$i=0,\ldots,r$. 
\end{cor}

\subsection{Mixed Hodge Structure on $\Rep_{\Gamma}^{0}G$ for Nilpotent $\Gamma$.}

The lower central series of a group $\Gamma$ is defined inductively
by $\Gamma_{1}:=\Gamma$, and $\Gamma_{i+1}:=[\Gamma,\Gamma_{i}]$
for $i>1$. A group $\Gamma$ is \textit{nilpotent} if the lower central
series terminates to the trivial group.

Let $\Gamma$ be a finitely generated nilpotent group. It is a general
theorem that all finitely generated nilpotent groups are finitely
presentable (and residually finite); see \cite{Hi}.

Recall that the abelianization of $\Gamma$: 
\[
\Gamma_{Ab}:=\Gamma/[\Gamma,\Gamma],
\]
can be written as $Ab(\Gamma)\simeq\mathbb{Z}^{r}\oplus F$ where
$r\in\mathbb{N}_{\geq0}$ is the abelian rank of $\Gamma$, and $F$
is a finite abelian group. We now generalize some of the previous
results to nilpotent groups. 
\begin{thm}
\label{thm-mhs-hom} Let $\Gamma$ be a finitely generated nilpotent
group with abelian rank $r\geq1$. Then, the algebraic variety $\mathcal{R}_{\Gamma}^{0}G=\hom^{0}(\Gamma,G)$
has dimension $\dim G+(r-1)\dim T$ and its MHS coincides with the
MHS on $(G/T)\times_{W}T^{r}$. 
\end{thm}

\begin{proof}
By \cite{BS}, we have $\hom^{0}(\Gamma_{Ab},K)\cong\hom^{0}(\Gamma,K)$
and consequently, from \cite{Be}, $\hom^{0}(\Gamma_{Ab},G)$ is homotopic
to $\hom^{0}(\Gamma,G)$. From \cite{Ba}, we know that 
\[
\varphi_{K}:(K/T_{K})\times_{W}T_{K}^{r}\to\hom^{0}(\Gamma_{Ab},K)
\]
is a birational surjection (in fact, a desingularization) that induces
an isomorphism in cohomology. This map is defined by $[(gT,t_{1},...,t_{r})]_{W}\mapsto(gt_{1}g^{-1},...,gt_{r}g^{-1})$,
and we can likewise define $\varphi_{G}$ in the complex situation.

These maps come together to form the following commutative diagram:
\begin{equation}
\xymatrix{(G/T)\times_{W}T^{r}\ar[r]^{\varphi_{G}} & \hom^{0}(\Gamma_{Ab},G)\ar[r] & \hom^{0}(\Gamma,G)\\
(K/T_{K})\times_{W}T_{K}^{r}\ar@{^{(}->}[u]\ar[r]^{\varphi_{K}} & \hom^{0}(\Gamma_{Ab},K)\ar@{^{(}->}[u]\ar[r]^{\cong} & \hom^{0}(\Gamma,K).\ar@{^{(}->}[u]
}
\label{key-ladder}
\end{equation}

Since the bottom row induces isomorphisms in cohomology, by commutativity,
all maps induce isomorphisms in cohomology. Since the upper row is
formed by algebraic maps, these induce isomorphisms of mixed Hodge
structures of the respective cohomologies. The dimension formula is
clear since $\dim G/T=\dim G-\dim T$. 
\end{proof}
\begin{rem}
Let $\Gamma_{Ab}\cong\Z^{r}$ with free abelian generators $\gamma_{1},...,\gamma_{r}$.
We note some properties of the map $\varphi_{G}:(G/T)\times_{W}T^{r}\to\hom^{0}(\Z^{r},G)$.
Let $G_{ss}$ be the set of semisimple elements of $G$ (elements
in $G$ with closed conjugation orbits), and $\hom^{0}(\Z^{r},G_{ss}):=\{\rho\in\hom^{0}(\Z^{r},G)\ |\ \rho(\gamma_{i})\in G_{ss},\ 1\leq i\leq r\}$.
It is shown in \cite{FL2} that $\hom^{0}(\Z^{r},G_{ss})$ is exactly
the set of representations with closed conjugation orbits. In the
identity component, these are exactly the representations whose image
can be conjugated to a fixed maximal torus. Hence, the image of $\varphi_{G}$
is exactly the set $\hom^{0}(\Z^{r},G_{ss})$. So we see that $\varphi_{G}$
is not surjective and $\hom^{0}(\Z^{r},G_{ss})$ is a constructible
set (not obvious a priori). Now from \cite{PeSo} we know that $\hom^{0}(\Z^{r},G)$
is homotopic to $\hom^{0}(\Z^{r},G_{ss})$, and since $G_{ss}$ is
dense in $G$ we deduce that $\hom^{0}(\Z^{r},G_{ss})$ is dense in
$\hom^{0}(\Z^{r},G)$. Thus, $\varphi_{G}$ is dominant, homotopically
surjective, and induces a cohomological isomorphism (from Diagram
\eqref{key-ladder}). Since $W$ acts freely, $(G/T)\times_{W}T^{r}$
is smooth although $\hom^{0}(\Z^{r},G)$ is generally singular. By
Remark \ref{rem:smooth} below, the Zariski dense representations
in $\hom^{0}(\Z^{r},G)$ are smooth points. It is easy to see that
$\varphi_{G}^{-1}(\rho)$ is a point if $\rho$ is Zariski dense (a
generic condition). Hence, $\varphi_{G}$ is birational, although,
unlike its compact analogue $\varphi_{K}$, it is not a desingularization. 
\end{rem}

\begin{cor}
\label{cor-main-1} Let $G$ be a reductive $\C$-group of rank $m$, whose Weyl group has characteristic degrees $d_{1},\ldots,d_{m}$. Let $\Gamma$
be a finitely generated nilpotent group of abelian rank $r\geq1$.
The variety $\mathcal{R}_{\Gamma}^{0}G$ is of Hodge-Tate type and
its mixed Hodge polynomial is given by: 
\begin{equation}
\mu_{\mathcal{R}_{\Gamma}^{0}G}\left(t,u,v\right)=\frac{1}{|W|}\prod_{i=1}^{m}(1-(t^{2}uv)^{d_{i}})\sum_{g\in W}\,\frac{\det\left(I+tuv\,A_{g}\right)^{r}}{\det\left(I-t^{2}uv\,A_{g}\right)}.\label{eq:mu-hom-formula}
\end{equation}
\end{cor}

\begin{proof}
This follows immediately from Theorem \ref{thm:SrG} and Theorem \ref{thm-mhs-hom}. 
\end{proof}
Corollaries \ref{cor-main-1} and \ref{cor-main-2} together establish
Theorem \ref{thm:mu-formula} from the Introduction. 
\begin{cor}
\label{cor:Euler-char} For every finitely generated nilpotent group
$\Gamma$, and reductive $\mathbb{C}$-group $G$, the Poincaré polynomial
and $E$-polynomial of $\mathcal{R}_{\Gamma}^{0}G$ are given, respectively,
by: 
\begin{eqnarray*}
P_{t}\left(\mathcal{R}_{\Gamma}^{0}G\right) & = & \frac{1}{|W|}\prod_{i=1}^{m}(1-t^{2d_{i}})\sum_{g\in W}\,\frac{\det\left(I+t\,A_{g}\right)^{r}}{\det\left(I-t^{2}\,A_{g}\right)}\\
E_{\mathcal{R}_{\Gamma}^{0}G}(u,v) & = & \frac{1}{|W|}\prod_{i=1}^{m}(1-(uv)^{d_{i}})\sum_{g\in W}\,\det\left(I-uv\,A_{g}\right)^{r-1}
\end{eqnarray*}
and the Euler characteristic of $\mathcal{R}_{\Gamma}^{0}G$ vanishes. 
If the abelian rank of $\Gamma$ is $2$ and $G$ is simply-connected, then the $E$-polynomial simplifies
to $E_{\mathcal{R}_{\Gamma}^{0}G}(u,v)=\prod_{i=1}^{m}(1-(uv)^{d_{i}})$.
\end{cor}

\begin{proof}
This follows by evaluating Formula \eqref{eq:mu-hom-formula}
at $u=v=1$ for the Poincaré polynomial, and at $t=-1$ for the $E$-polynomial.
Then, the Euler characteristic is obtained as $\chi(\mathcal{R}_{\Gamma}^{0}G)=E_{\mathcal{R}_{\Gamma}^{0}G}(1,1)=0$, as $r\geq1$. 

Finally, when $r=2$ the term $(1/|W|)\sum_{g\in W}\det(I-uvA_{g})$
is equal to $\mu_{\mathcal{M}_{\mathbb{Z}}^{0}G}(-1,u,v)$ by
Theorem  \ref{thm-abelian-quot} below. On the other hand, since $G$ is simply-connected,
$\mathcal{M}_{\mathbb{Z}}^{0}G\cong T/W\cong\mathbb{C}^{\dim T}$, by a result of Steinberg
in \cite{Ste}. Being affine space, its $E$-polynomial equals 1 (see \cite[Example 2.6]{FS}), as wanted.
\end{proof}

\subsection{Some computations for classical groups}

For certain classes of groups, such as $G=\SL(n,\mathbb{C})$ and
$G=\GL(n,\mathbb{C})$, the above formulas can be made more explicit.
These cases have Weyl group $S_{n}$, the symmetric group on $n$
letters. In the $\GL(n,\mathbb{C})$ case, the action of a permutation
$\sigma\in S_{n}$ on the dual of the Cartan subalgebra of $\mathfrak{gl}_{n}$
can be identified with the action on $\mathbb{C}^{n}$ by permuting
the canonical basis vectors. Therefore, $\det(I-\lambda A_{\sigma})=\prod_{j=1}^{n}(1-\lambda^{j})^{\sigma_{j}}$,
where $\sigma\in S_{n}$ is a permutation with exactly $\sigma_{j}\geq0$
cycles of size $j\in\{1,\ldots,n\}$ (see, for example, \cite[Thm. 5.13]{FS}).
The collection $(\sigma_{1},\sigma_{2}\ldots,\sigma_{n})$ defines
a partition of $n$, one with exactly $\sigma_{j}$ parts of length
$j$, and the number of permutations $\sigma\in S_{n}$ with this
cycle pattern is (see \cite[1.3.2]{St}): 
\[
m_{\sigma}=n!\,{\textstyle (\prod_{j=1}^{n}\sigma_{j}!\,j^{\sigma_{j}}})^{-1}.
\]
Since the characteristic degrees of the Weyl group of $\GL(n,\mathbb{C})$ are exactly
$1,2,\ldots,n$, this leads to the following explicit formula: 
\[
\mu_{\mathcal{R}_{\Gamma}^{0}\GL(n,\mathbb{C})}\left(t,u,v\right)=\prod_{i=1}^{m}(1-(t^{2}uv)^{i})\sum_{\pi\vdash n}\,\prod_{j=1}^{n}\frac{(1-(-tuv)^{j})^{\pi_{j}r}}{{\textstyle \pi_{j}!\,j^{\pi_{j}}}(1-(t^{2}uv)^{j})^{\pi_{j}}},
\]
where $\pi\vdash n$ denotes a partition of $n$ with $\pi_{j}$ parts
of size $j$.

Moreover, in the $\GL(n,\mathbb{C})$ case, we can also derive a recursion
relation, which completely avoids the determination of partitions
or permutations. Since $\mu_{\mathcal{R}_{\Gamma}^{0}\GL(n,\mathbb{C})}$
depends only on $tuv$ and $t^{2}uv$, we use the substitutions $x=tuv$,
and $w=tx=t^{2}uv$. 
\begin{prop}
\label{prop:recursion}Let $G=\GL(n,\mathbb{C})$ and write $\mu_{n}^{r}(x,w):=\mu_{\mathcal{R}_{\Gamma}^{0}G}(t,u,v)$
for a nilpotent group $\Gamma$, of abelian rank $r\geq1$. Then,
we have the recursion relation: 
\begin{equation}
\mu_{n}^{r}(x,w)=\frac{1}{n}\sum_{k=1}^{n}f((-x)^{k},w^{k})\,c_{k}(w)\,\mu_{n-k}^{r}(x,w),\label{eq:recursion}
\end{equation}
with $f(x,w):=\frac{(1-x)^{r}}{1-w}$ and $c_{k}(w):=\prod_{i=0}^{k-1}(1-w^{n-i})$. 
\end{prop}

\begin{proof}
For fixed $r\in\mathbb{N}$, let $\phi_{n}(z,w)$ be the rational
function in variables $z,w$, defined by: 
\[
\phi_{n}(z,w):=\frac{1}{n!}\sum_{g\in S_{n}}\,\frac{\det\left(I-z\,A_{g}\right)^{r}}{\det\left(I-w\,A_{g}\right)},
\]
with $\phi_{0}(z,w)\equiv1$. By \cite[Thm 3.1]{Fl}, the generating
series for $\phi_{n}(z,w)$ is a so-called plethystic exponential:
\[
1+\sum_{n\geq1}\phi_{n}(z,w)\,y^{n}=\mathsf{PE}(f(z,w)\,y):=\exp\left(\sum_{k\geq1}f(z^{k},w^{k})\,\frac{y^{k}}{k}\right)
\]
with $f(z,w)=\frac{(1-z)^{r}}{1-w}$. Differentiating the above identity
with respect to $y$ we get: 
\[
\sum_{n\geq1}n\phi_{n}(z,w)\,y^{n-1}=\left(1+\sum_{m\geq1}\phi_{m}(z,w)\,y^{m}\right)\left(\sum_{k\geq1}f(z^{k},w^{k})\,y^{k-1}\right),
\]
which, by picking the coefficient of $y^{n}$, leads to the recurrence:
\begin{equation}
\phi_{n}(z,w)=\frac{1}{n}\sum_{k=1}^{n}f(z^{k},w^{k})\,\phi_{n-k}(z,w).\label{eq:phi}
\end{equation}
To apply this to $\mu_{n}^{r}(x,w)$ we use Equation \eqref{eq:mu-hom-formula}
in the form: 
\[
\phi_{n}(-x,w)=\frac{\mu_{n}^{r}(x,w)}{\prod_{i=1}^{n}(1-w^{i})},
\]
so the wanted recurrence follows by replacing $z=-x$ in Equation
\eqref{eq:phi}. 
\end{proof}
In the $\SL(n,\mathbb{C})$ case, also with Weyl group $S_{n}$, the
action is the same permutation action, but restricted to the vector
subspace of $\mathbb{C}^{n}$ whose coordinates add up to zero. Hence,
the formula for $\det(I-\lambda A_{\pi})$ acting on dual of $\mathfrak{sl}_{n}$
is now: 
\begin{equation}
\det(I-\lambda A_{\pi}):=\frac{1}{1-\lambda}\prod_{j=1}^{n}(1-\lambda^{j})^{\pi_{j}},\label{eq:det}
\end{equation}
for a permutation $\sigma\in S_{n}$ with $\sigma_{j}$ cycles of
size $j$. Recalling that $\SL(n,\mathbb{C})$ is a group of rank
$n-1$ whose Weyl group has characteristic degrees $2,3,\ldots,n$, we derive the
following formula, reflecting the fact that the $\mathcal{R}_{\Gamma}^{0}\GL(n,\mathbb{C})$
and $\mathcal{R}_{\Gamma}^{0}\SL(n,\mathbb{C})$ cases only differ
by a torus. 
\begin{cor}
Let $G=\SL(n,\mathbb{C})$ and $\Gamma$ be a finitely generated nilpotent
group of abelian rank $r\geq1$. Then: 
\begin{equation}
\mu_{\mathcal{R}_{\Gamma}^{0}\SL(n,\mathbb{C})}\left(x,w\right)=\frac{1}{(1+x)^{r}}\,\mu_{\mathcal{R}_{\Gamma}^{0}\GL(n,\mathbb{C})}\left(x,w\right).\label{eq:mu-hom-GLn}
\end{equation}
\end{cor}

\begin{rem}
The recursion formulae in \eqref{eq:recursion} and \eqref{eq:mu-hom-GLn}
have been implemented in a \textit{Mathematica} notebook available
on \cite{gcode}.
\end{rem}

\begin{example}
From \eqref{eq:recursion} and \eqref{eq:mu-hom-GLn} we can quickly
write down the first few cases for $\SL(n,\mathbb{C})$. To obtain
$\mu_{\mathcal{R}_{\Gamma}^{0}\SL(n,\mathbb{C})}(t,u,v)$ one just
needs to substitute $x=tuv$ and $w=t^{2}uv$.
\begin{eqnarray*}
\mu_{\mathcal{R}_{\Gamma}^{0}\SL(2,\mathbb{C})} & = & {\textstyle \frac{1}{2}}\left((1+w)(1+x)^{r}+(1-w)(1-x)^{r}\right).\\
\mu_{\mathcal{R}_{\Gamma}^{0}\SL(3,\mathbb{C})} & = & {\textstyle \frac{1}{6}}(1+2w+2w^{2}+w^{3})(1+x)^{2r}+{\textstyle \frac{1}{2}}(1-w^{3})(1-x^{2})^{r}\\
 &  & +{\textstyle \frac{1}{3}}(1-w-w^{2}+w^{3})(1-x+x^{2})^{r}.\\
\mu_{\mathcal{R}_{\Gamma}^{0}\SL(4,\mathbb{C})} & = & {\textstyle \frac{1}{24}}(1+w)(1+w+w^{2})(1+w+w^{2}+w^{3})(1+x)^{3r}\\
 &  & +{\textstyle \frac{1}{4}}(1+w+w^{2})(1-w^{4})(1+x)^{r}(1-x^{2})^{r}\\
 &  & +{\textstyle \frac{1}{8}}(1-w^{3})(1-w+w^{2}-w^{3})(1-x)^{r}(1-x^{2})^{2r}\\
 &  & +{\textstyle \frac{1}{3}}(1-w^{2})(1-w^{4})(1+x^{3})^{r}\\
 &  & +{\textstyle \frac{1}{4}}(1-w)(1-w^{2})(1-w^{3})(1+x+x^{2}+x^{3})^{r}.
\end{eqnarray*}
Putting $x=t$ and $w=t^{2}$ we recover the expressions for the Poincaré
polynomial in \cite{Ba} and \cite{RS}.\footnote{Note that our first term for $n=3$ corrects the corresponding term
in \cite[pg. 749]{Ba}.} Note that with $x=-1$, $w=1$ we confirm the vanishing of the Euler
characteristic. With $w=-x$ we get formulas for the $E$-polynomial,
and with $x=w=1$ (that is, $t=u=v=1$) we get: 
\[
\mu_{\mathcal{R}_{\Gamma}^{0}\SL(n,\mathbb{C})}\left(1,1,1\right)=2^{(n-1)r},
\]
the dimension of the total cohomology of $T^{r}$, confirming that
$H^{*}(\mathcal{R}_{\Gamma}^{0}\SL(n,\mathbb{C}))$ is a regrading
of $H^{*}(T^{r})$. 
\end{example}

\begin{example}
Consider now the group $G=\Sp(2n,\mathbb{C})$ which has rank $n$
and dimension $n(2n+1)$. Its Weyl group is the so-called \emph{hyperoctahedral
group}: the group of symmetries of the hypercube of dimension $n$,
denoted $C_{n}$, of order $|C_{n}|=2^{n}n!$. It can be described
as the subgroup of permutations of the set $S_{\pm n}:=\{-n,\ldots,-1,1,\ldots,n\}$
satisfying: 
\[
\sigma\in C_{n}\subset S_{\pm n}\quad\quad\Longleftrightarrow\quad\quad\sigma(-i)=-\sigma(i)\quad\forall1\leq i\leq n.
\]
The action of $g\in C_{n}$ on the dual of the Lie algebra $\mathfrak{sp}_{2n}\cong\mathbb{C}^{n}$
is the following natural action. If we denote by $e_{1},\ldots,e_{n}$
the standard basis of $\mathbb{C}^{n}$, and let $e_{-i}:=-e_{i}$,
then $g\cdot e_{i}=e_{\sigma(i)}$, for all $1\leq i\leq n$, where
$g\in C_{n}$ corresponds to the permutation $\sigma\in S_{\pm n}$.

Given that $\Sp(2,\mathbb{C})\cong\SL(2,\mathbb{C})$, we consider
the next case: $n=2$. $\Sp(4,\mathbb{C})$  has complex dimension 10, and
its Weyl group is $C_{2}$,
which is known to be isomorphic to the dihedral group of order $8$ (the
symmetries of the square): 
\[
C_{2}=\{e,a,a^{2},a^{3},ba,ba^{2},ba^{3}\},
\]
where $a$ acts by counter-clockwise rotation of $\frac{\pi}{2}$
(that is $e_{1}\mapsto e_{2}\mapsto-e_{1}\mapsto-e_{2}\mapsto e_{1}$)
and $b$ is the reflection along the first coordinate axis ($e_{1}\mapsto e_{1}$
and $e_{2}\mapsto-e_{2}$). Then, we have: 
\[
a=\left(\begin{array}{cc}
0 & -1\\
1 & 0
\end{array}\right),\quad\quad b=\left(\begin{array}{cc}
1 & 0\\
0 & -1
\end{array}\right)
\]
and simple computations give the following table, with $p_{g}(\lambda)=\det(I-\lambda A_{g})$.

\vspace{2mm}
 $\quad\quad\quad\quad\quad\quad\quad\quad\quad\quad\quad$
\begin{tabular}{c|c}
$g\in C_{2}$  & $p_{g}(\lambda)$\tabularnewline
\hline 
$e$  & $(1-\lambda)^{2}$\tabularnewline
$a,a^{3}$  & $1+\lambda^{2}$\tabularnewline
$a^{2}$  & $(1+\lambda)^{2}$\tabularnewline
$b,ba,ba^{2},ba^{3}$  & $1-\lambda^{2}$\tabularnewline
\end{tabular}

From this, since the characteristic degrees of $\C_{2}$
are $2,4$, we compute, using again $x=tuv$ and $w=t^{2}uv$: 
\begin{eqnarray*}
\mu_{\mathcal{R}_{\Gamma}^{0}\Sp(4,\mathbb{C})} & = & \frac{1}{2^{2}2!}(1-w^{2})(1-w^{4})\sum_{g\in C_{2}}\,\frac{p_{g}(-x)^{r}}{p_{g}(w)}\\
 & = & {\textstyle \frac{1}{8}}(1-w^{2})(1-w^{4})\left({\textstyle \frac{(1+x)^{2r}}{(1-w)^{2}}+2\frac{(1+x^{2})^{r}}{1+w^{2}}+\frac{(1-x)^{2r}}{(1+w)^{2}}+4\frac{(1-x^{2})^{r}}{1-w^{2}}}\right)\\
 & = & {\textstyle \frac{1}{8}}(1+w)(1+w+w^{2}+w^{3})(1+x)^{2r}+{\textstyle \frac{1}{4}}(1-w^{2})^{2}(1+x^{2})^{r}\\
 &  & +{\textstyle \frac{1}{8}}(1-w)(1-w+w^{2}-w^{3})(1-x)^{2r}+{\textstyle \frac{1}{2}}(1-w^{4})(1-x^{2})^{r}.
\end{eqnarray*}
Again, we note that with $x=t$ and $w=t^{2}$ we obtain the Poincaré polynomial. Setting $w=x=1$ we obtain $2^{2r}$, and setting $w=1=-x$ we get zero, both as expected. The above formula gives a new result even for 
${\mathcal{R}_{\mathbb{Z}^{2}}\Sp(4,\mathbb{C})}$, 
the 12 dimensional variety of pairs of commuting $\Sp(4,\mathbb{C})$ matrices. 
Indeed, we obtain the following Poincaré polynomial
\begin{eqnarray*}
P_{\mathcal{R}_{\mathbb{Z}^{2}}\Sp(4,\mathbb{C})}(t) = 
1+ t^2 + t^4 + 2(t^3+t^5+t^6+t^7+t^9) + 3t^{10},
\end{eqnarray*}
and $E$-polynomial $E_{\mathcal{R}_{\mathbb{Z}^{2}}\Sp(4,\mathbb{C})}(u,v)=(1-(uv)^2)(1-(uv)^4)$, as expected from Corollary
\ref{cor:Euler-char}
\end{example}

\subsection{$G$-Equivariant Cohomology of $\Rep_{\Gamma}^{0}(G)$ }

For a Lie group $G$, denote $G$-equivariant cohomology (over
$\C$) by $H_{G}$. We now resume our main setup: $G$ is a reductive
$\C$-group, $K$ is a maximal compact subgroup of $G$, $T$ is a
maximal torus in $G$ and $T_{K}$ is a compatible maximal torus in
$K$ (so $T_{K}=T\cap K$). Again, let $\Gamma$ be a finitely generated
nilpotent group of abelian rank $r\geq1$, so the torsion free part
of its abelianization is $\Z^{r}$.

Since $G$ and $K$ are homotopic, as are $\Rep_{\Gamma}^{0}(G)$
and $\Rep_{\Gamma}^0(K)$, we conclude there is an isomorphism in equivariant cohomology:
\[
H_{G}^{*}(\Rep_{\Gamma}^{0}(G))\cong H_{K}^{*}(\Rep_{\Gamma}^{0}(K)).
\]

Then, from Baird's thesis \cite{BaT}, precisely pages 39 and 55,
and Corollary 7.4.4, the $G$-equivariant and $K$-equivariant maps
in Diagram \eqref{key-ladder} imply we have the following isomorphisms:
\begin{eqnarray*}
H_{K}^{*}(\Rep_{\Gamma}^{0}(K)) & \cong & H_{K}^{*}(\Rep_{\Z^{r}}^{0}(K))\\
 & \cong & H_{K}^{*}((K/T_{K})\times T_{K}^{r})^{W}\\
 & \cong & H_{T_{K}}^{*}(T_{K}^{r})^{W}\\
 & \cong & [H^{*}(T_{K}^{r})\otimes H^{*}(BT_{K})]^{W}\\
 & \cong & [H^{*}(T^{r})\otimes H^{*}(BT)]^{W}.
\end{eqnarray*}

We have already computed the Hilbert series of this latter ring in Equation \eqref{eq:hilbert}. Thus we conclude: 
\begin{cor}
There is a MHS on the $G$-equivariant cohomology of $\Rep_{\Gamma}^{0}(G)$
and the $G$-equivariant mixed Hodge series is: 
\[
\mu_{\Rep_{\Gamma}^{0}(G)}^{G}=\frac{1}{|W|}\sum_{g\in W}\,\frac{\det\left(I+tuv\,A_{g}\right)^{r}}{\det\left(I-t^{2}uv\,A_{g}\right)}.
\]
\end{cor}

\section{Mixed Hodge Structure on \texorpdfstring{$\hom^{0}(\Gamma,G)\quot G$}{Hom(F,G)//G}}

\label{sec:mhs-char}

Now we prove the statements in Theorems \ref{thm:MHS} and \ref{thm:mu-formula}
on the connected component $\mathcal{M}_{\Gamma}^{0}G$ of the trivial
representation of the character variety $\mathcal{M}_{\Gamma}G=\hom(\Gamma,G)\quot G$.

We start with the free abelian case, $\Gamma\cong\mathbb{Z}^{r}$,
noting a number of corollaries to Theorem \ref{thm:Tequals0}. 
\begin{cor}
\label{cor-normalization} $\mathcal{M}_{\mathbb{Z}^{r}}^{0}G$ is
irreducible, and there exists a birational bijective morphism 
\[
\chi:T^{r}/W\to\mathcal{M}_{\mathbb{Z}^{r}}^{0}G
\]
which is the normalization map. In particular, we have equality of Grothendieck classes: $[T^{r}/W]=[\M_{\Z^{r}}^{0}G]$. 
\end{cor}

\begin{proof}
As noted earlier, $\mathcal{M}_{\mathbb{Z}^{r}}^{T}G$ is irreducible,
and we have shown that $\mathcal{M}_{\mathbb{Z}^{r}}^{0}G=\mathcal{M}_{\mathbb{Z}^{r}}^{T}G$.
We also know from \cite{Sik} that there is a bijective birational
morphism $T^{r}/W\to\mathcal{M}_{\mathbb{Z}^{r}}^{T}G$. The first sentence
follows since $T^{r}/W$ is normal (since the GIT quotient of a normal
variety is normal). Since $\chi$ is a bijective map, the statement  on Grothendieck classes follows from \cite[Page 115]{BB} (see also \cite{G}).
\end{proof}
We will say that a reductive $\C$-group $G$ is of \textit{classical
type} if its derived subgroup $DG$ admits a central isogeny by a
product of groups of type $\SL(n,\C)$, $\mathrm{Sp}(2n,\C)$, or
$\mathrm{SO}(n,\C)$ for varying $n$ (not necessarily all the same
$n$ within the product). 
\begin{cor}
\label{cor-classical} If $G$ is of classical type, then $\mathcal{M}_{\mathbb{Z}^{r}}^{0}G$
is normal and $\chi:T^{r}/W\to\mathcal{M}_{\mathbb{Z}^{r}}^{0}G$
is an isomorphism. 
\end{cor}

\begin{proof}
Given $\mathcal{M}_{\mathbb{Z}^{r}}^{0}G=\mathcal{M}_{\mathbb{Z}^{r}}^{T}G$
this follows from \cite{FS}. Here is a sketch of the result in \cite{FS}.
Sikora showed the result for $\SL(n,\C)$, $\mathrm{Sp}(2n,\C)$,
or $\mathrm{SO}(n,\C)$ in \cite{Sik}. It is trivially true for tori.
In general, $\M_{\Z^{r}}^{0}(G\times H)\cong\M_{\Z^{r}}^{0}G\times\M_{\Z^{r}}^{0}H$
and also $\M_{\Z^{r}}^{0}(G/F)\cong(\M_{\Z^{r}}^{0}G)/F^{r},$ for
finite central subgroups $F$. The result then follows from the central
isogeny theorem for reductive $\C$-groups and the facts that GIT
quotients of normal varieties are normal, and cartesian products of
normal varieties are normal. 
\end{proof}
Since $\mathcal{M}_{\mathbb{Z}^{r}}^{0}G=\mathcal{M}_{\mathbb{Z}^{r}}^{T}G$
we know for any $[\rho]\in\M_{\Z^{r}}^{0}G$ its image is contained
in some maximal torus which we may assume is $T$. We will say such
a representation is \textit{Zariski dense} if its image is Zariski
dense in $T$. We note that every representation in the identity component
is \textit{reducible}; that is, its image is contained in a proper
parabolic subgroup of $G$. For many choices of $\Gamma$, reducible
representations are singular points; see for example \cite{FL3,GLR}.
The next corollary is in contrast to this. 
\begin{cor}
\label{cor-smooth} Assume $r\geq2$, and that $[\rho]\in\mathcal{M}_{\mathbb{Z}^{r}}^{0}G$
is Zariski dense. Then 
\begin{enumerate}
\item $[\rho]$ is a smooth point, and 
\item the map $\chi:T^{r}/W\to\M_{\Z^{r}}^{0}G$ is étale at $[\rho]$. 
\end{enumerate}
\end{cor}

\begin{proof}
Since $\mathcal{M}_{\mathbb{Z}^{r}}^{0}G=\mathcal{M}_{\mathbb{Z}^{r}}^{T}G$,
and \cite[Theorem 4.1]{Sik} shows that if $[\rho]\in\M_{\Z^{r}}^{T}G$
and is Zariski dense then (1) holds on $\mathcal{M}_{\mathbb{Z}^{r}}^{T}G$,
(1) is also true for $\mathcal{M}_{\mathbb{Z}^{r}}^{0}G$. For (2),
\cite[Theorem 4.1]{Sik} shows that the map induces an isomorphism
of tangent spaces on the torus component when $\rho$ is Zariski dense;
this implies the map is étale at $[\rho]$ by (1). 
\end{proof}
\begin{rem}
If $r=1$, then we have $G\quot G\cong T/W$ and is smooth if $DG$
is simply-connected by \cite{Ste} and \cite[Proposition 3.1]{Bo}.
The converse is not true however, since $\mathrm{PSL}(2,\C)\quot\mathrm{PSL}(2,\C)\cong\C$
is smooth. 
\end{rem}

The map $\chi:T^{r}/W\to\mathcal{M}_{\mathbb{Z}^{r}}^{0}G$ is the
normalization map in general and it is an open question whether or
not it is an isomorphism in general \cite{Sik}. We note that $\chi$
is an isomorphism if and only if $\chi$ is étale and that holds if
and only if $\mathcal{M}_{\mathbb{Z}^{r}}^{0}G$ is normal. 

\begin{cor}
\label{cor-sing} Let $G$ be of classical type. Then the singular
locus of $\mathcal{M}_{\mathbb{Z}^{r}}^{0}G$ is of orbifold type;
that is, consists only of finite quotient singularities. 
\end{cor}

\begin{proof}
In the case that $G$ is of classical type we know that $\chi$ is
an isomorphism since $\mathcal{M}_{\mathbb{Z}^{r}}^{0}G$ is normal.
Thus, the singular locus of $\M_{\Z^{r}}^{0}G$ is exactly the singular
locus of $T^{r}/W$. Since $T^{r}/W$ is the finite quotient of a
manifold, the result follows. 
\end{proof}
\begin{rem}
\label{rem:smooth}From this point-of-view, we can see easily why
the Zariski dense representations are smooth. The Zariski dense representations
are tuples $(t_{1},...,t_{r})$ that generate a Zariski dense subgroup
of $T$ (most $t_{i}$'s do this by themselves). If $w\cdot\rho=\rho$
then $w\cdot\rho(\gamma)=\rho(\gamma)$ for all $\gamma$. Since $\rho$
is Zariski dense we conclude that $w\cdot t=t$ for all $t\in T$.
We conclude $w=1$ and so $W$ acts freely on the set of Zariski dense
representations. This shows they are smooth points and the singular
locus is contained in the non-Zariski dense representations. 
\end{rem}

\begin{rem}
Assume $r\geq2$. If $\rho$ is not Zariski dense, then the identity
component of $A:=\overline{\rho(\Z^{r})}$ is, up to conjugation, a proper subtorus of
$T$. It seems reasonable to suppose that $A$ is contained in the
fixed locus of a non-trivial $w\in W$. The fixed loci $(T^{r})^{w}$
for $w\not=1$ are of codimension greater than 1 since $r\geq2$ and
$(T^{w})^{0}$ is a proper subtorus. So, in light of the Shephard-Todd
Theorem \cite{ShTo}, it appears likely that the non-Zariski dense
representations are exactly the singular locus (for $r\geq2$). 
\end{rem}

\subsection{The MHS on $\mathcal{M}_{\Gamma}^{0}G$}

Given an isomorphism of groups $\varphi:\Gamma_{1}\to\Gamma_{2}$
there exists a (contravariant) biregular morphism $\varphi^{*}:\M_{\Gamma_{2}}G\to\M_{\Gamma_{1}}G$
given by $\varphi^{*}([\rho])=[\rho\circ\varphi]$ with inverse $(\varphi^{-1})^{*}$.
Consequently, the topology of $\M_{\Gamma}G$ and its mixed Hodge
structure (MHS) are independent of the presentation of $\Gamma$.
Hence, the same holds for $\mathcal{M}_{\Gamma}^{0}G$, for any $\Gamma$.

Let us start with the free abelian case, $\Gamma\cong\mathbb{Z}^{r}$,
where we know that $\mathcal{M}_{\mathbb{Z}^{r}}^{0}G=\mathcal{M}_{\mathbb{Z}^{r}}^{T}G$. 
\begin{thm}
\label{thm-abelian-quot} Let $G$ be a reductive $\C$-group, $T$
a maximal torus, and $W$ the Weyl group. Then, the MHS of $\mathcal{M}_{\mathbb{Z}^{r}}^{0}G$
coincides with the one of $T^{r}/W$ and its mixed Hodge polynomial
is given by: 
\begin{eqnarray}
\mu_{\mathcal{M}_{\mathbb{Z}^{r}}^{0}G}\left(t,u,v\right) & = & \frac{1}{|W|}\sum_{g\in W}\left[\det\left(I+tuv\,A_{g}\right)\right]^{r},\label{eq:mu-general}
\end{eqnarray}
where $A_{g}$ is the automorphism induced on $H^{1}(T,\C)$ by $g\in W$,
and $I$ is the identity automorphism. 
\end{thm}

\begin{proof}
We have the following commutative diagram with vertical arrows being
strong deformation retractions from \cite{FL2}: 
\begin{eqnarray*}
\xymatrix{T_{K}^{r}/W\ar@{^{(}->}[d]\ar[r]^{\cong} & \mathcal{N}_{\mathbb{Z}^{r}}^{0}K\ar@{^{(}->}[d]\\
T^{r}/W\ar[r]^{\chi} & \mathcal{M}_{\mathbb{Z}^{r}}^{0}G.
}
\end{eqnarray*}

Thus, $\chi$ induces isomorphisms in cohomology and since it is an
algebraic map, these isomorphisms preserve mixed Hodge structures.
Thus, the MHS on $T^{r}/W$ and on $\mathcal{M}_{\mathbb{Z}^{r}}^{0}G$
coincide. The formula then follows immediately from \cite{FS}. 
\end{proof}
\begin{thm}
\label{thm-nilpotent-mhs-quot} Let $\Gamma$ be a finitely generated
nilpotent group of abelian rank $r\geq1$. The MHS on $\mathcal{M}_{\Gamma}^{0}G$
coincides with the MHS on $T^{r}/W$. 
\end{thm}

\begin{proof}
This follows from \cite[Corollary 1.4]{BS}, where they prove isomorphisms
in cohomology given by algebraic maps. 
\end{proof}

\begin{cor}
\label{cor-main-2} Let $\Gamma$ be a finitely generated nilpotent
group of abelian rank $r\geq1$. Then, for all reductive $\C$-groups
$G$ we have: 
\begin{eqnarray*}
\mu_{\mathcal{M}_{\Gamma}^{0}G}\left(t,u,v\right) & = & \frac{1}{|W|}\sum_{g\in W}\left[\det\left(I+tuv\,A_{g}\right)\right]^{r}.
\end{eqnarray*}
\end{cor}

\begin{proof}
This follows directly from Theorem \ref{thm-abelian-quot} and Theorem
\ref{thm-nilpotent-mhs-quot}. 
\end{proof}

Note that the $G$-equivariant cohomology of the moduli space $\M_{\Gamma}^{0}(G)$
is the usual cohomology since the $G$-action is trivial on $\M_{\Gamma}^{0}(G)$.

\begin{cor}
\label{thm-csepoly}The compactly supported mixed Hodge polynomial
of $\M_{\Z^{r}}^{0}(G)$ is: 
\[
\mu_{\M_{\Z^{r}}^{0}(G)}^{c}(t,u,v)=\frac{t^{r\dim T}}{|W|}\sum_{g\in W}\left[\det(tuvI+A_{g})\right]^{r}.
\]
\end{cor}

\begin{proof}
If a variety $X$ of (complex) dimension $d$ satisfies Poincaré duality,
then $\mu_{X}$ and $\mu_{X}^{c}$ are related by: $\mu_{X}^{c}(t,u,v)=(t^{2}uv)^{d}\mu_{X}(t^{-1},u^{-1},v^{-1});$ see \cite[Remark $3.10(1)$]{FS}. From Theorem \ref{thm-abelian-quot},
the mixed Hodge structures and polynomials of $\M_{\Z^{r}}^{0}(G)$
and $T^{r}/W$ coincide. Since $T^{r}/W$ is an orbifold, it satisfies
Poincaré duality for MHSs (see \cite[Section 4.2]{FS}). Hence, using $d=\dim T$, we have:
\begin{eqnarray*}
\mu_{T^{r}/W}^{c}(t,u,v) & = & (t^{2}uv)^{rd}\mu_{T^{r}/W}(t^{-1},u^{-1},v^{-1})\\
 & = & (t^{2}uv)^{rd}\frac{1}{|W|}\sum_{g\in W}\left[\det\left(I+\frac{1}{tuv}\,A_{g}\right)\right]^{r}\\
 & = & \frac{t^{rd}}{|W|}\sum_{g\in W}\left[(tuv)^{d}\det\left(I+\frac{1}{tuv}\,A_{g}\right)\right]^{r}
\end{eqnarray*}
as claimed, since $A_{g}$ are automorphisms of the Lie algebra of
$T$.
\end{proof}

\subsection{Examples for classical groups}

As in the case of representation varieties, the character varieties
for $\GL(n,\mathbb{C})$ and $\SL(n,\mathbb{C})$ also allow closed
expressions in terms of partitions $\pi$ of $n$. In \cite[Thm 5.13]{FS},
it was shown that $\mathcal{M}_{\mathbb{Z}^{r}}\GL(n,\mathbb{C})$
has round cohomology, so that $x=tuv$ is the only relevant variable,
and that: 
\[
\mu_{\mathcal{M}_{\mathbb{Z}^{r}}\GL(n,\mathbb{C})}\left(x\right)=\mu_{\mathcal{M}_{\mathbb{Z}^{r}}\SL(n,\mathbb{C})}\left(x\right)\,(1+x)^{r}.
\]
From the present analysis, the same formulas work also for the identity
components of the character varieties of any nilpotent group $\Gamma$
with abelianization $\mathbb{Z}^{r}$. Moreover, we can also obtain
a recurrence relation as follows. 
\begin{prop}
Let $G=\GL(n,\mathbb{C})$ and write $\nu_{n}^{r}(x):=\mu_{\mathcal{M}_{\Gamma}^{0}G}(t,u,v)$
for a nilpotent group $\Gamma$, of abelian rank $r$. Then, with
$h(x):=(1-x)^{r}$, we have: 
\begin{equation}
\nu_{n}^{r}(x)=\frac{1}{n}\sum_{k=1}^{n}h((-x)^{k})\,\nu_{n-k}^{r}(x).\label{eq:recursion-1}
\end{equation}
\end{prop}

\begin{proof}
As in Proposition \ref{prop:recursion}, define $\psi_{n}(z)$ to
be the rational function of $z$: 
\[
\psi_{n}(z):=\frac{1}{n!}\sum_{g\in S_{n}}\det\left(I-z\,A_{g}\right)^{r}=\nu_{n}^{r}\left(-z\right),
\]
with $\psi_{0}(z)\equiv1$. By \cite[Thm 3.1]{Fl}, the generating
series for the $\psi_{n}(z)$ is now the plethystic exponential $1+\sum_{n\geq1}\psi_{n}(z)\,y^{n}=\mathsf{PE}(h(z)\,y)$
with $h(z)=(1-z)^{r}$. As before, the derivative with respect to
$y$ now gives: 
\[
\sum_{n\geq1}n\psi_{n}(z)\,y^{n-1}=\left(1+\sum_{m\geq1}\psi_{m}(z)\,y^{m}\right)\left(\sum_{k\geq1}h(z^{k})\,y^{k-1}\right),
\]
which, by picking the coefficient of $y^{n}$, leads to the recurrence:
\begin{equation}
\psi_{n}(z)=\frac{1}{n}\sum_{k=1}^{n}h(z^{k})\,\psi_{n-k}(z),\label{eq:psi}
\end{equation}
so the proposition follows by replacing $z=-x$ in Equation \eqref{eq:psi}. 
\end{proof}

\subsection{Point count over finite fields and compactly supported $E$-polynomials}
\label{subsec:poly-count}

In this subsection, we show that our formulae for mixed Hodge polynomials
also compute the number of points of the identity component of character varieties of free abelian groups over finite fields.

Let $X$ be a separated scheme of finite type over $\mathbb{Z}$.
We say that $X$ is \emph{polynomial count}, with counting polynomial
$\mathcal{P}_{X}(x)\in\mathbb{Z}[x]$ if for all but finitely many
primes $p$, and finite fields $\mathbb{F}_{q}$ with $q=p^{k}$, its number of $\mathbb{F}_{q}$-points is given
by $\#X(\mathbb{F}_{q})=\mathcal{P}{}_{X}(q)$. By extension of scalars,
we can consider the varieties $X(\mathbb{C})$ and $X(\overline{\mathbb{F}}_{q})$,
respectively, over $\mathbb{C}$ and over the algebraic closure $\overline{\mathbb{F}}_{q}$.

Next, consider the $k$-th compactly supported $l$-adic cohomology
of $X(\overline{\mathbb{F}}_{q})$, denoted $H_{c}^{k}(X(\overline{\mathbb{F}}_{q}),\overline{\mathbb{Q}}_{l})$ for a prime $l$ with $\gcd(l,q)=1$, and the Frobenius morphism $$F:X(\overline{\mathbb{F}}_{q})\to X(\overline{\mathbb{F}}_{q}),$$
whose fixed points are precisely the $\mathbb{F}_{q}$ points of $X$. 

Following Dimca-Lehrer, we say that $X$ is \emph{minimally pure} if $X$ is irreducible of dimension $n$, and $F$ acts on $H_{c}^{k}(X(\overline{\mathbb{F}}_{q}),\overline{\mathbb{Q}}_{l})$ with all eigenvalues equal to $q^{k-n}$ (see \cite[Definition 5.1]{DiLe}) (This notion is the analogue of round, for $X$ smooth over $\C$).

Now consider the $\mathbb{Z}$-scheme $X=Spec(\mathbb{Z}[x_{i},y_{i}]/(x_{i}y_{i}-1))$,
with $2n$ variables $x_{i}$ and $y_{i}$, whose complex variety is
a torus $X_{\mathbb{C}}=(\mathbb{C}^{*})^{n}$. According to \cite[Thm. 5.4]{DiLe}, $X$ is minimally pure.

\begin{thm}\label{thm:torus-count}
Fix $r\in\mathbb{N}$, and a prime power $q$, and let $T_{\mathbb{Z}}$ be a $\mathbb{Z}$-scheme such that $T_{\mathbb{C}}$ is a maximal torus of a reductive ${\mathbb{C}}$-group with Weyl group $W$. Then, the quotient of $T^r_\Z$ by the diagonal action of $W$, denoted $T_{\mathbb{Z}}^{r}/W$, is polynomial count and we have:
\[
\#(T_{\mathbb{Z}}^{r}/W)(\mathbb{F}_{q})=\frac{1}{|W|}\sum_{g\in W}\left[\det\left(qI-\,A_{g}\right)\right]^{r}.
\]
\end{thm}

\begin{proof}
We apply the following general result. If $H$ is a finite group acting
on a $\mathbb{Z}$-scheme $Y$, we have:
$$H_{c}^{k}((Y/H)(\overline{\mathbb{F}}_{q}),\overline{\mathbb{Q}}_{l})\cong H_{c}^{k}(Y(\overline{\mathbb{F}}_{q}),\overline{\mathbb{Q}}_{l})^{H},$$
as in \cite[proof of Prop. 5.5]{DiLe}. Now let $Y=T_{\mathbb{Z}}^{r}$
and $H=W$. Since $T_{\mathbb{Z}}^{r}$ is minimally pure, the Frobenius
morphism acts on $V_{q}^{k}:=H_{c}^{k}(T_{\mathbb{Z}}^{r}(\overline{\mathbb{F}}_{q}),\overline{\mathbb{Q}}_{l})^{W}$ with all eigenvalues equal to $q^{k-d}$, $d=\dim T_{\mathbb{Z}}^{r}$.

Since $(T_{\mathbb{Z}}^{r}/W)(\mathbb{F}_{q})$ consists precisely
of the Frobenius fixed points of $(T_{\mathbb{Z}}^{r}/W)(\overline{\mathbb{F}}_{q})$,
we apply Grothendieck's fixed point formula (see for example \cite[Equation $(5.3.1)$]{DiLe}), to obtain:
$$
\#(T_{\mathbb{Z}}^{r}/W)(\mathbb{F}_{q})=\sum_{k=0}^{2d}(-1)^{k}\mathrm{Tr}(F,\,V_{q}^{k}).
$$ 
This is a polynomial in $q$ since $\mathrm{Tr}(F,\,V_{q}^{k})$, the trace of $F$ on
$V_{q}^{k}$, is a sum of powers
of $q$. Hence, by Katz's theorem \cite{HaRo} and \cite{Se}, the compactly supported $E$-polynomial of $(T_{\mathbb{Z}}^{r}/W)_{\mathbb{C}}=T_{\mathbb{C}}^{r}/W$ coincides with the counting polynomial, with $q=uv$. The required formula then comes from the compactly supported mixed Hodge polynomial in Corollary \ref{thm-csepoly}, setting $t=-1$ and $q=uv$.
\end{proof}

\begin{cor}
The counting polynomial of $\M_{\Z^{r}}^{0}(G)$ is 
\[
\mathcal{P}_{\M_{\Z^{r}}^{0}(G)}(x)=\frac{1}{|W|}\sum_{g\in W}\left[\det(xI-A_{g})\right]^{r}.
\]
\end{cor}

\begin{proof}

By Corollary \ref{cor-normalization}, $\M_{\Z^{r}}^{0}(G)$ is bijective to $T^{r}/W$ via the normalization map $\chi$.  Moreover, $\chi$ is induced by the inclusion $T^r\hookrightarrow \Rep_{\Z^{r}}^{0}(G)$ and so is defined over the ring of integers (with the possible inversion of a finite number of primes).
Thus the counting functions of these two varieties are equal and so the result follows from Theorem \ref{thm:torus-count}.
\end{proof}

\begin{rem}
The smooth model for $\Rep_{\Z^{r}}^{0}(G)$,
namely $(G/T)\times_{W}T^{r}$, is also polynomial count since $W$
acts freely on $(G/T)\times T^{r}$, the product of polynomial count
varieties is polynomial count, and $T^{r}$ and $G/T$ are polynomial
count by \cite{Br}. However, this does not show $\Rep_{\Z^{r}}^{0}(G)$
is polynomial count since the map relating $\Rep_{\Z^{r}}^{0}(G)$
and $(G/T)\times_{W}T^{r}$, although a cohomological isomorphism,
is neither injective nor surjective. 
\end{rem}

\section{Exotic Components}\label{sec:exotic}

In this last section we describe the MHS on the full character varieties
of $\Z^{r}$ (not only the identity component) in special cases described
in \cite{ACG}. Let $p$ be a prime integer, and $\Z_{p}$ be the
cyclic group or order $p$. We will use the same notation for center
of $\SU(p)$ which is realized as the subgroup of scalar matrices
with values $p$-th roots of unity. Let $\Delta(p)$ be the diagonal
of $\Z_{p}$ in $(\Z_{p})^{m}=Z(\SU(p)^{m})$. Let $K_{m,p}:=\SU(p)^{m}/\Delta(p)$.
For example, $K_{1,p}=\mathrm{PU}(p)$, the projective unitary group.

In \cite{ACG}, all the components in $\hom(\Z^{r},K_{m,p})$ and
$\M_{\Z^{r}}K_{m,p}$ are described. In particular, $\M_{\Z^{r}}K_{m,p}$
consists of the identity component 
\[
\M_{\Z^{r}}^{0}K_{m,p}\cong(S^{1})^{(p-1)rm}/(S_{p})^{m}
\]
where $S_{p}$ is the symmetric group on $p$ letters, and 
\[
N(p,m):=\frac{p^{(m-1)(r-2)}(p^{r}-1)(p^{r-1}-1)}{p^{2}-1}
\]
discrete points.

There is a one-to-one correspondence between the isolated points in
$\M_{\Z^{r}}K_{m,p}$ and non-identity path-components in $\hom(\Z^{r},K_{m,p})$.
Each such path-component is isomorphic to the homogeneous space $\SU(p)^{m}/(\Z_{p}^{m-1}\times E_{p})$
where $E_{p}\subset\SU(p)$ is isomorphic to the quaternion group
$Q_{8}$ if $p$ is even and the group of triangular $3\times3$ matrices
over the $\Z_{p}$, with 1's on the diagonal when $p$ is odd (either
way it is ``extra-special'' of order $p^{3}$). 
\begin{cor}
\label{cor:exotic-Hom}Let $G_{m,p}$ be the complexification of $K_{m,p}$:
$G_{m,p}\cong\SL(p,\C)^{m}/\Delta(p)$. Then: $\mu_{\mathcal{R}_{\Z^{r}}G_{m,p}}\left(t,u,v\right)=$
\[
=\left(\frac{1}{p!}\prod_{i=2}^{p}(1-(t^{2}uv)^{i})\sum_{g\in S_{p}}\frac{\det(I+tuv\,A_{g})^{r}}{\det(I-t^{2}uv\,A_{g})}\right)^{m}+N(p,m)\prod_{j=2}^{p}\left(1+t^{2j-1}u^{j}v^{j}\right)^{m}.
\]
\end{cor}

\begin{proof}
The Weyl group $W$ of $G_{m,p}$ is a direct product of $m$ copies
of $S_{p}$: the Weyl group of $\SL(p,\C)$, and its action on the
(dual of the) Lie algebra of maximal torus, is the product action.
Therefore, using again $x=tuv$ and $w=t^{2}uv$, for the identity
component $\mathcal{R}_{\Z^{r}}^{0}G_{m,p}$ we have: 
\begin{eqnarray*}
\mu_{\mathcal{R}_{\Z^{r}}^{0}G_{m,p}}\left(t,u,v\right) & = & \left(\frac{1}{p!}\prod_{i=2}^{p}(1-w^{i})\right)^{m}\!\!\!\!\!\sum_{(g_{1},\cdots,g_{m})\in W}\!\!\frac{\det(I_{p}+x\,A_{g_{1}})^{r}\cdots\det(I_{p}+x\,A_{g_{m}})^{r}}{\det(I_{p}-w\,A_{g_{1}})\cdots\det(I-w\,A_{g_{m}})}\\
 & = & \left(\frac{1}{p!}\prod_{i=2}^{p}(1-w^{i})\right)^{m}\left(\sum_{g\in S_{p}}\frac{\det(I+x\,A_{g})^{r}}{\det(I-w\,A_{g})}\right)^{m}.
\end{eqnarray*}
Now, since the MHS on $\SL(p,\C)^{m}/(\Z_{p}^{m-1}\times E_{p})$
coincides with that of $\SL(p,\C)^{m}$ by Lemma \ref{lem:MHS}, each
of the $N(p,m)$ components, other than $\mathcal{R}_{\Z^{r}}^{0}G_{m,p}$,
contributes $\mu(\SL(p,\C)^{m})=\mu(\SL(p,\C))^{m}=\prod_{j=2}^{p}\left(1+t^{2j-1}u^{j}v^{j}\right)^{m}$. 
\end{proof}
For the character variety, from the fact that each isolated point
adds a constant 1, we have the following corollary: 
\begin{cor}
\label{cor:exotic-CV}The mixed Hodge polynomial of $\mathcal{M}_{\Z^{r}}G_{m,p}$
is: 
\begin{eqnarray*}
\mu_{\mathcal{M}_{\Z^{r}}G_{m,p}}\left(t,u,v\right) & = & \left(\frac{1}{p!}\sum_{g\in S_{p}}\det\left(I+tuv\,A_{g}\right)^{r}\right)^{m}+N(p,m).
\end{eqnarray*}
\end{cor}

\begin{proof}
The same argument of Corollary \ref{cor:exotic-Hom}, implies that
the identity component verifies: $\mu_{\mathcal{M}_{\Z^{r}}^{0}G_{m,p}}=(\mu_{\mathcal{M}_{\Z^{r}}\SL(p,\C)})^{m}$,
so the formula is clear. 
\end{proof}
\begin{rem}
Given two reductive groups $G$ and $H$, both the $(G\times H)$-representation
varieties and the $(G\times H)$-character varieties are cartesian
products of the $G$ and $H$ varieties: 
\[
\mathcal{R}_{\Gamma}(G\times H)=\mathcal{R}_{\Gamma}G\times\mathcal{R}_{\Gamma}H,\quad\quad\mathcal{M}_{\Gamma}(G\times H)=\mathcal{M}_{\Gamma}G\times\mathcal{M}_{\Gamma}H.
\]
From Corollaries \ref{cor:exotic-Hom} and \ref{cor:exotic-CV}, the
mixed Hodge polynomial of the identity components of these $G_{m,p}$-character
varieties behaves multiplicatively, even though $\mathcal{R}_{\Z^{r}}^{0}G_{m,p}$
and $\mathcal{M}_{\Z^{r}}^{0}G_{m,p}$ are not cartesian products. 
\end{rem}

\begin{rem}
Since $G$ is connected, $\mathcal{R}_{\Z^{r}}^{0}G=\mathcal{R}_{\Z^{r}}G$
if and only if $\mathcal{M}_{\Z^{r}}^{0}G=\mathcal{M}_{\Z^{r}}G$.
And by \cite{FL2}, $\mathcal{M}_{\Z^{r}}^{0}G=\mathcal{M}_{\Z^{r}}G$
if and only if (a) $r=1$, or (b) $r=2$ and $G$ is simply-connected,
or (c) $r\geq3$ and $G$ is a product of $\SL(n,\C)$'s and $\Sp(n,\C)$'s. 
\end{rem}

\subsection*{Statements \& Declarations}

The authors have no competing interests to declare that are relevant to the content of this article.

\end{document}